\documentclass[a4paper,12pt,intlimits,oneside]{amsart}

\usepackage{enumerate}
\usepackage{amsfonts,amsmath}

\usepackage{latexsym,amssymb,amsthm}

\newcommand{\comment}[1]{}

\newcounter{rea}
\newcounter{rek}
\newcounter{res}
\begin{document}

\title[New weighted multilinear operators and commutators]{New weighted multilinear operators and commutators of Hardy-Ces\`{a}ro type}         
\author{Ha Duy HUNG}    
\address{High School for Gifted Students,
Hanoi National University of Education, 136 Xuan Thuy, Hanoi, Vietnam} 
\email{{\tt hunghaduy@gmail.com}}
\author{Luong Dang KY}
\address{Department of Mathematics, Quy Nhon University, 
170 An Duong Vuong, Quy Nhon, Binh Dinh, Viet Nam} 
\email{{\tt dangky@math.cnrs.fr}}


\begin{abstract}
A general class of weighted multilinear Hardy-Ces\`{a}ro operators that acts on the product of Lebesgue spaces and central Morrey spaces. Their sharp bounds are also obtained. In addition, we obtain sufficient and necessary conditions on weight functions so that the commutators of these weighted multilinear Hardy-Ces\`{a}ro operators (with symbols in central BMO space) are bounded on the product of central Morrey spaces. These results extends known results on multilinear Hardy operators.
\end{abstract}

\maketitle
\newtheorem{theorem}{Theorem}[section]
\newtheorem{lemma}{Lemma}[section]
\newtheorem{proposition}{Proposition}[section]
\newtheorem{remark}{Remark}[section]
\newtheorem{corollary}{Corollary}[section]
\newtheorem{definition}{Definition}[section]
\newtheorem{example}{Example}[section]
\numberwithin{equation}{section}
\newtheorem{Theorem}{Theorem}[section]
\newtheorem{Lemma}{Lemma}[section]
\newtheorem{Proposition}{Proposition}[section]
\newtheorem{Remark}{Remark}[section]
\newtheorem{Corollary}{Corollary}[section]
\newtheorem{Definition}{Definition}[section]
\newtheorem{Example}{Example}[section]

\section{Introduction} 
The Hardy integral inequality and its variants play an important role in various branches of analysis such as approximation theory, differential equations, theory of function spaces etc (see \cite{bradley, edmunds, grafakos,hardy,muckenhoupt3, sawyer} and references therein). Therefore, there are various papers studying Hardy integral inequalities for operator $S$ and its generalizations. The classical Hardy operator, its variants and extensions were appeared in various papers (we refer to \cite{bradley,carton,fu1,fu2,fu3,hardy, hung,muckenhoupt3, sawyer,xiao} for surveys and historical details about these differerent aspects of the subject). On the other hand, the study of multilinear operators is not motivated by a mere quest to generalize the theory of linear operators but rather by their natural appearance in analysis. Coifman and Meyer in their pioneer work in the 1970s were one of the first to adopt a multilinear point of view in their study of certain singular integral operators, such as the Calder\'{o}n commutators, paraproducts, and pseudodifferential operators. 
\vskip12pt
 Let $\psi:[0,1]\to[0,\infty)$ be a measurable function. The weighted Hardy operator $U_\psi$ is defined on all complex-valued measurable functions $f$ on $\mathbb R^d$ as 
\begin{equation}\label{sec1eq1}
U_\psi f(x)=\int_0^1f(tx)\psi(t)dt.
\end{equation}
When $\psi=1$, this operator is reduced to the usual Hardy operator $S$ defined by $Sf(x)=\frac1x\int_0^x f(t)dt$. Results on the boundedness of $U_\psi$  on $L^p\left(\mathbb R^d\right)$ were first proved by Carton-Lebrun and Fosset \cite{carton}. Under certain conditions on $\psi$, the authors \cite{carton} found that $U_\psi$ is bounded from $BMO(\mathbb R^d)$ into itself. Furthermore,  $U_\psi$ commutes with the Hilbert transform in the case $n=1$ and with a certain Calder\'{o}n-Zygmund singular integral operator (and thus with the Riesz transform) in the case $n\geq2$. A deeper extension of the results obtained in  \cite{carton} was due to Jie Xiao \cite{xiao}.  
\begin{theorem}\label{sec1theo1}(\cite{xiao})
Let $1<p<\infty$ and $\psi:[0,1]\to[0,\infty)$ be a measurable function. Then, $U_\psi$ is bounded on $L^p(\mathbb R^d)$ if and only if 
\begin{equation}\label{sec1eq2}
\int_{0}^1t^{-n/p}\psi(t)dt<\infty.
\end{equation}
Furthermore, 
\begin{equation}\label{sec1eq3}
\|U_\psi\|_{L^p(\mathbb R^d)\to L^p(\mathbb R^d)}=\int_{0}^1t^{-n/p}\psi(t)dt<\infty.
\end{equation}
\end{theorem}
Theorem \ref{sec1theo1} implies immediately the following celebrated integral inequality, due to Hardy \cite{hardy}
\begin{equation}\label{sec1eq4}
\|Sf\|_{L^p(\mathbb R)}\leq \frac p{p-1}\|f\|_{L^p(\mathbb R)}.
\end{equation}
For further applications of Theorem \ref{sec1theo1}, for examples the sharp bounds of classical Riemann-Liouville integral operator, see \cite{fu1,fu2}.
Recently, Chuong and Hung \cite{chuong1} introduced a more general form of $U_\psi$ operator  as follows.
 \begin{equation}\label{sec1eq5}
U_{\psi,s}f(x)=\int_0^1f\left(s(t) x\right)\psi(t)dt.
\end{equation}
Here $\psi:[0,1]\to[0,\infty)$, $s:[0,1]\to\mathbb R$ are measurable functions and $f$ is a measurable complex valued function on $\mathbb R^d$. The authors in \cite{chuong1} obtained sharp bounds of $U_{\psi,s}$ on weighted Lebesgue and BMO spaces, where weights are of homogeneous type. A characterization on  weight functions so that the commutator of $U_{\psi,s}$ is bounded on Lebesgue spaces with symbols in $BMO$. We notice that also in \cite{chuong1}, the authors proved the boundedness on Lebesgue spaces of the following operator (see Theorem 3.2 \cite{chuong1})
\[
\mathcal H_{\psi,s}f(x)=\int_0^\infty f\left(s(t)x\right)\psi(t)dt.
\]
When $d=1$ and $s(t)=\frac1t$, then $\mathcal H_{\psi,s}$ reduces to the classical weighted Hausdorff operator (see \cite{liflyand,lerner} and references therein).
\vskip12pt
Very recently, the multilinear version of $U_\psi$ operators was introduced by Z.W. Fu, S.L Gong, S.Z. Lu and W. Yuan \cite{fu2}. They defined the weighted multilinear Hardy operator as 
\begin{equation}\label{sec1eq6}
U^{m}_{\psi}\left(f_1,\ldots,f_m\right)=\int_{0<t_1,\ldots,t_m<1}\left(\prod_{i=1}^mf_i\left(t_ix\right)\right)\psi(t_1,\ldots,t_m)dt_1\ldots dt_m.
\end{equation}
As showed in \cite{fu2}, when $d=1$ and $\psi(t_1,\ldots,t_m)=\frac1{\Gamma(\alpha)\left|(1-t_1,\ldots,1-t_m)\right|^{m-\alpha}}$, where $\alpha\in(0;m)$, then $U^m_\psi(f_1,\ldots,f_m)(x)=|x|^\alpha I_\alpha^m(f_1,\ldots,f_m)(x)$. The operator $I^m_\alpha$ turns out to be the one-sided analogous to the one-dimensional multilinear Riesz operator ${\mathcal J}^m_\alpha$ studied by Kenig and Stein \cite{kenig}, where
\[
{\mathcal J}^m_\alpha(f_1,\ldots,f_m)(x)=\int_{t_1,\ldots,t_m\in \mathbb R}\frac{\prod_{k=1}^mf_k(t_k)}{\left|(x-t_1,\ldots,x-t_m)\right|^{m-\alpha}}d_1\ldots dt_m.
\]
In \cite{fu2}, the authors obtain the sharp bounds of $U^m_{\psi}$ on the product of Lebesgue spaces and central Morrey spaces. They also proved sufficient and necessary conditions of the weight functions so that the commutators of $U^m_\psi$, with symbols in central BMO spaces, are bounded on the product of central Morrey spaces. In \cite{gong}, S. Gong, Z. Fu, and B. Ma studied the bounds of $U^m_\psi$ on the product Herz spaces and the product Morrey-
Herz spaces. Notice that there is another multilinear version of Hardy operators which was also studied in  \cite{fu3,lu2}.
\vskip12pt
Motivated from \cite{chuong1,fu2,fu1,gong,hung}, this paper aims to study the boundedness of a more general multilinear operator of Hardy type as follows.
\begin{definition}Let $m,n\in\mathbb N$, $\psi:[0,1]^n\to[0,\infty)$, $s_1,\ldots,s_m:[0,1]^n\to\mathbb R$. Given $f_1,\ldots,f_m:\mathbb R^d\to\mathbb C$ be measurable functions, we define the weighted multilinear Hardy-Ces\`{a}ro operator $U^{m,n}_{\psi,\overrightarrow{s}}$ by
\begin{equation}\label{sec1eq7}
U^{m,n}_{\psi,\overrightarrow{s}}\left(f_1,\ldots,f_m\right)(x):=\int_{[0,1]^n}\left(\prod\limits_{k=1}^mf_k\left(s_k(t)x\right)\right)\psi(t)dt,
\end{equation}
where $\overrightarrow{s}=(s_1,\ldots,s_m)$.
\vskip12pt
A mutilinear version of $\mathcal H_{\psi,s}$ can be defined as 
\begin{equation}\label{sec1eq7a}
\mathcal H^{m,n}_{\psi,\overrightarrow{s}}\left(f_1,\ldots,f_m\right)(x):=\int_{\mathbb R_+^n}\left(\prod\limits_{k=1}^mf_k\left(s_k(t)x\right)\right)\psi(t)dt,
\end{equation}
where $\mathbb R_+$ is the set of all positive real numbers.
\end{definition}

It is obviously that, when $n=m$ and $s_k(t)=t_k$, $U^{m,n}_{\psi,\overrightarrow{s}}$ is reduced to $U^m_{\psi}$ as defined in (\ref{sec1eq5}). The main aim of the paper is to establish the sharp bounds of $U^{m,n}_{\psi,\overrightarrow{s}}$ and $\mathcal H^{m,n}_{\psi,\overrightarrow{s}}$ on the product of weighted Lebesgue spaces and weighted central Morrey spaces, with weights of homogeneous types. In addition, we prove sufficient and and  necessary conditions of the weight functions so that commutators of such weighted multilinear Hardy-Ces\'{a}ro operators (with symbols in $\lambda$-central BMO space) are bounded on the product of central Morrey spaces. Since $U^{m,n}_{\psi,\overrightarrow{s}}$ is trivially more general than known operators $U_\psi$, $U^m_\psi$, our results can be used to recover main results of \cite{chuong1,fu1, fu2,xiao}.

 \vskip12pt
Throughout the whole paper, the letter $C$ will indicate an absolute constant, probably different at different occurrences. With $\chi_E$ we will denote the characteristic function of a set $E$ and $B(x,r)$ will be a ball centered at $x$ with radius $r$. With $|A|$ we will denote the Lebesgue measure of a measurable set $E$, and $E^c$ will be the set $\mathbb R^d\setminus E$. With $\omega(E)$ we will denote by $\int_E \omega(x)dx$ and $S_d$ will be the unit ball $\{x\in\mathbb R^d:\;|x|=1\}$.

\section{Notations and definitions}
Throughout this paper, $\omega(x)$ will be denote a nonnegative measurable function on $\mathbb R^d$. Let us recall  that a measurable function $f$ belongs to $L^p_\omega\left(\mathbb R^d\right)$ if
\begin{equation}\label{sec2eq1}
\|f\|_{L^p_\omega}=\left(\int_{\mathbb R^d}|f(x)|^p\omega(x)dx\right)^{1/p}<\infty.
\end{equation}
The weighted $BMO$ space $BMO(\omega)$ is defined as the set of all functions $f$, which are of bounded mean oscillation with weight $\omega$, that is,
\begin{equation}\label{sec2eq2}
\|f\|_{BMO(\omega)}=\sup\limits_B\frac1{\omega(B)}\int_B|f(x)-f_{B,\omega}|\omega(x)dx<\infty,
\end{equation}
where supremum is taken over all the balls $B\subset \mathbb R^d$. Here  $f_{B,\omega}$ is the mean value of $f$ on $B$ with weight $\omega$:
\[
f_{B,\omega}=\frac1{\omega(B)}\int_Bf(x)\omega(x)dx.
\]
The case $\omega\equiv1$ of (\ref{sec2eq2}) corresponds to the class of functions of bounded mean oscillation of F. John and L. Nirenberg \cite{john}. We obverse that $L^\infty(\mathbb R^d)\subset BMO(\omega)$. 
\vskip12pt
Next we recall the definition of Morrey spaces. It is well-known that Morrey spaces are useful to study the local behavior of solutions to second-order elliptic partial differential equations and the boundedness of Hardy-Littlewood maximal operator, the fractional integral operators, singular integral operators (see \cite{adam,chiarenza,komori}). We notice that the weighted Morrey spaces were first introduced by Komori and Shirai \cite{komori}, where they used them to study the boundedness of some important classical operators in harmonic analysis like as Hardy-Littlewood maximal operator, Calder\'{o}n-Zygmund operators.  
\begin{definition}\label{sec4def1} Let $\lambda \in\mathbb R$, $1\leq p<\infty$ and $\omega$ be an weight function. The weighted Morrey space $L^{p,\lambda}_{\omega}(\mathbb R^d)$ is defined by the set of all locally $p-$integrable functions $f$ satisfying 
\begin{equation}\label{sec2eq3}
\|f\|_{L^{p,\lambda}_{\omega}(\mathbb R^d)}=\sup_{a\in\mathbb R^d,R>0}\left(\frac1{\omega\left(B(a,R)\right)^{1+\lambda p}}\int_{B(0,R)}|f(x)|^p\omega(x)dx\right)^{1/p}<\infty.
\end{equation} 
\end{definition}
The spaces of bounded central mean oscillation $CMO^q\left(\mathbb R^d\right)$ appears naturally when considering the dual spaces of the homogeneous Herz type Hardy spaces and were introduced by Lu and Yang (see \cite{lu1}). The relationships between central BMO spaces and Morrey spaces were studied by Alvarez, Guzm\'{a}n-Partida and Lakey \cite{alvarez}. Furthermore, they introduced $\lambda-$central BMO spaces and central Morrey spaces as follows.  
\begin{definition}\label{sec4def2} Let $\lambda \in\mathbb R$ and $1<p<\infty$. The weighted central Morrey space $\dot{B}^{p,\lambda}_{\omega}(\mathbb R^d)$ is defined by the set of all locally $p-$integrable functions $f$ satisfying 
\begin{equation}\label{sec2eq4}
\|f\|_{\dot{B}^{p,\lambda}_{\omega}(\mathbb R^d)}=\sup_{R>0}\left(\frac1{\omega\left(B(0,R)\right)^{1+\lambda p}}\int_{B(0,R)}|f(x)|^p\omega(x)dx\right)^{1/p}<\infty.
\end{equation} 
\end{definition}
Obviously, $\dot{B}^{p,\lambda}_{\omega}(\mathbb R^d)$ is  a Banach space and one can easily check that $\dot{B}^{p,\lambda}_{\omega}(\mathbb R^d)=\{0\}$ if $\lambda<-\frac1q$. Similar to the classical Morrey space, we only consider the case $-1/p\leq\lambda<0$.
\vskip12pt

\begin{definition}\label{sec2def4} Let $\lambda<\frac{1}{n}$ and $1<q<\infty$ be two real numbers. The weighted  space $CMO^{q,\lambda}_{\omega}\left(\mathbb R^d \right)$ is defined as the set of all function $f\in L^q_{\omega,loc}\left(\mathbb R^d \right)$ such that 
\begin{equation}\label{sec2eq8}
\|f\|_{{CMO}^{q,\lambda}_{\omega}\left(\mathbb R^d \right)} =\sup\limits_{R>0}\left(\frac{1}{\omega\left(B(0,R)\right)^{1+\lambda q}}\int_{B(0,R)}|f(x)-f_{B(0,R),\omega}|^q\omega(x)dx\right)^{1/q}<\infty.
\end{equation}
\end{definition}
\vskip12pt
Let $\rho$ be the measure on $(0,\infty)$ so that $\rho(E)=\int_E r^{d-1}dr$ and the map $\Phi(x)=\left(|x|,\frac x{|x|}\right)$. Then there exists an unique Borel measure $\sigma$ on $S_n$ such that $\rho\times\sigma$ is the Borel measure induced by $\Phi$ from Lebesgue measure on $\mathbb R^d$ ($d>1$). (see \cite[page 78]{folland} for more details). In one dimension case, it it conventional that $\int_{S_d}\omega(x)d\sigma(x)$ refers to $2\omega(1)$.  
\begin{definition}
Let $\alpha$ be a real number. Let $\mathcal W_{\alpha}$ be the set of all functions $\omega$ on $\mathbb R^d$, which are measurable, $\omega(x)>0$ for almost everywhere $x\in\mathbb R^d$, $0<\int_{S_d}\omega(y)d\sigma(y)<\infty$, and are absolutely homogeneous of degree $\alpha$, that is $\omega(tx)=|t|^\alpha\omega(x)$, for all $t\in\mathbb R\setminus\{0\}$, $x\in\mathbb R^d$.
\end{definition}
\vskip12pt
Let us describe some typical examples and properties of $\mathcal W_\alpha$. Note that, a weight $\omega\in\mathcal W_\alpha$ may not need to belong to  $L^1_{\rm loc}(\mathbb R^d)$. In fact, we observe that if $\omega\in \mathcal W_\alpha$, then $\omega\in L^1_{\rm loc}(\mathbb R^d)$ if and only if $\alpha>-d$.  If $d=1$, then $\omega(x)= c|x|^\alpha$, for some positive constant $c$. For $d\geq1$ and $\alpha\neq0$, $\omega(x)=|x|^\alpha$ is in $\mathcal W_\alpha$. If $\omega_1,\omega_2$ is in $\mathcal W_\alpha$, so is $\theta\omega_1+\lambda\omega_2$ for all $\theta,\lambda>0$. There are many other examples in case $d>1$ and $\alpha\neq0$, namely $\omega(x_1,\ldots,x_d)=|x_1|^\alpha$. In case $d>1$ and $\alpha=0$, we can construct a non-trivial example as the following: let $\phi$ be any positive, even and locally integrable function on $S_d=\{x\in\mathbb R^d:\;|x|=1\}$, then
\[
\omega(x)=
\begin{cases} 
\phi\left(\dfrac x{|x|}\right)\qquad\text{if $x\neq0$},&\\
\quad0\quad\;\quad\qquad\text{if $x=0$},&\\
\end{cases}
\]
 is in $\mathcal W_0$.

\section{Sharp boundedness of $U^{m,n}_{\psi,\vec{s}}$ on the product of weighted Lebesgue spaces}
Before stating our results we give some notations and definitions.
\begin{definition}For $m$ exponents $1\leq p_j<\infty$, $j=1,\ldots,m$ and  $\alpha_1,\ldots,\alpha_m>-d$ we will often write $p$ for the number given by  
$\frac1p=\frac1{p_1}+\cdots+\frac1{p_m}$ and   
\begin{equation}\label{sec3eq1}
\alpha:=\frac{p\alpha_1}{p_1}+\cdots+\frac{p\alpha_m}{p_m}.
\end{equation}
\end{definition}
\begin{definition} Given $\omega_k\in\mathcal W_{\alpha_k}$, $k=1,\ldots,m$, set 
\begin{equation}\label{sec3eq2}
\omega(x):=\prod\limits_{k=1}^m\omega_k^{p/p_k}(x).
\end{equation}
It is obvious that $\omega\in\mathcal W_\alpha$.  We say that $(\omega_1,\ldots,\omega_m)$ satisfies the $\mathcal W_{\overrightarrow{\alpha}}$ condition if 
\begin{equation}\label{sec3eq3}
\omega(S_d)\geq \prod_{k=1}^m\omega_k(S_d)^{\frac p{p_k}}.
\end{equation}
\end{definition}
Notice that the weights $\omega$ as defined in (\ref{sec3eq2}) were used in \cite{grafakos2} to obtain multilinear extrapolation results. On the other hand, Condition (\ref{sec3eq3}) holds for power weights. In fact,  (\ref{sec3eq3}) becomes to equality for $\omega_k(x)=|x|^{\alpha_k}$, $k=1,\ldots,m$.
\vskip12pt
First we will show the following result which can be viewed as an extension of Theorem \ref{sec1theo1} to the multilinear case. 
\begin{theorem}\label{sec3theo1} Let $s_1,\ldots,s_m:[0,1]^n\to\mathbb R$ be measurable functions such that for every $k=1,\ldots,m$  then $|s_k(t_1,\ldots,t_n)|\geq \min\{t_1^\beta,\ldots,t_d^\beta\}$ almost everywhere $(t_1,\ldots,t_n)\in[0,1]^n$, $k=1,\ldots,m$, for some $\beta>0$. Let $1\leq p_k<\infty$, $k=1,\ldots,m$ and $\overrightarrow{s}=(s_1,\ldots,s_m)$. Assume that $(\omega_1,\ldots,\omega_m)$ satisfies $\mathcal W_{\overrightarrow{\alpha}}$ condition. Then $U^{m,n}_{\psi,\overrightarrow{s}}$ is bounded from $L^{p_1}_{\omega_1}\left(\mathbb R^d\right)\times\cdots\times L^{p_m}_{\omega_m}\left(\mathbb R^d\right)$ to $L^{p}_{\omega}\left(\mathbb R^d\right)$ if and only if 
\begin{equation}\label{maineq1}
{\mathcal A}=\int_{[0,1]^n}\left(\prod_{k=1}^m|s_k(t)|^{-\frac{d+\alpha_k}{p_k}}\right)\psi(t)dt<\infty.
\end{equation}
Furthermore, 
\begin{equation}\label{sec3eq4}
\left\|U^{m,n}_{\psi,\overrightarrow{s}}\right\|_{L^{p_1}_{\omega_1}\left(\mathbb R^d\right)\times\cdots\times L^{p_m}_{\omega_m}\left(\mathbb R^d\right)\to L^{p}_{\omega}\left(\mathbb R^d\right)}={\mathcal A}.
\end{equation}
\end{theorem}
\begin{proof}First we take $\alpha=\sum\limits_{k=1}^m\frac{p\alpha_k}{p_k}$ as in (\ref{sec3eq1}), then it is trivial that $\omega\in {\mathcal W}_\alpha$. Suppose that $\mathcal A$ is finite. Let $f_k\in L^{p_k}_{\omega_k}\left(\mathbb R^d\right)$. Applying Minkowski's inequality, H\"{o}lder's inequality and change of variable, we have  
\begin{align*}
\left\|U^{m,n}_{\psi,\overrightarrow{s}}\left(f_1,\ldots,f_m\right)\right\|_{L^{p}_\omega\left(\mathbb R^d\right)}\leq&\;\left( \int_{\mathbb R^d}\left(\int_{[0,1]^n}\prod\limits_{k=1}^m\left|f_k\left(s_k(t)x\right)\right|\psi(t)dt\right)^p\omega(x)dx\right)^{1/p}\\
\leq&\;\int_{[0,1]^n}\left(\int_{\mathbb R^d}\prod\limits_{k=1}^m\left|f_k\left(s_k(t)x\right)\right|^p\omega(x)dx\right)^{1/p}\psi(t)dt\\
\leq&\;\int_{[0,1]^n}\prod_{k=1}^m\left(\int_{\mathbb R^d}\prod\limits_{k=1}^m\left|f_k\left(s_k(t)x\right)\right|^{p_k}\omega_k(x)dx\right)^{1/{p_k}}\psi(t)dt\\
=&\;{\mathcal A}\cdot\prod_{k=1}^m\|f_k\|_{p_k,\omega_k}.
\end{align*}
The last inequality implies that $U^{m,n}_{\psi,\overrightarrow{s}}$ is bounded from $L^{p_1}_{\omega_1}\left(\mathbb R^d\right)\times\cdots\times L^{p_m}_{\omega_m}\left(\mathbb R^d\right)$ to $L^{p}_{\omega}\left(\mathbb R^d\right)$ and 
\begin{align}\label{sec3eq6}
\left\|U^{m,n}_{\psi,\overrightarrow{s}}\right\|_{L^{p_1}_{\omega_1}\left(\mathbb R^d\right)\times\cdots\times L^{p_m}_{\omega_m}\left(\mathbb R^d\right)\to L^{p}_{\omega}\left(\mathbb R^d\right)}\leq {\mathcal A}.
\end{align}
In order to prove the converse of the theorem, we first need the following lemma.
\begin{lemma}\label{mainlem1} Let $w\in {\mathcal W}_\alpha$, $\alpha>-d$ and $\varepsilon>0$. Then the function 
\[
f_{p,\varepsilon}(x)=
\begin{cases}
0\qquad\qquad\quad\;\;\text{if}\quad|x|\leq1&\\
|x|^{-\frac{d+\alpha}p-\varepsilon}\qquad\text{if}\quad|x|\geq1,
\end{cases}
\]
belongs to $L^p_w\left(\mathbb R^d\right)$ and $\|f_{p,\varepsilon}\|_{p,w}=\left(\frac{w(S_d)}{p\varepsilon}\right)^{1/p}$.
\end{lemma}
Since the proof of the lemma is straightforward, we omit it. Now we shall prove the theorem. Let $\varepsilon$ be  an arbitrary positive number and  for each $k=1,\ldots,m$ we set $\varepsilon_k=\frac{p\varepsilon}{p_k}$ and
\[
f_{p_k,\varepsilon_k}(x)=
\begin{cases}
0\qquad\qquad\qquad\;\text{if}\quad|x|\leq1&\\
|x|^{-\frac{d+\alpha_k}{p_k}-\varepsilon_k}\qquad\text{if}\quad|x|\geq1.
\end{cases}
\]
Lemma \ref{mainlem1} implies that $f_{p_k,\varepsilon_k}\in L^{p_k}_{\omega_k}\left(\mathbb R^d\right)$ and 
\[
\|f_{p_k,\varepsilon_k}\|_{p_k,\omega_k}=\left(\frac{\omega_k(S_d)}{p_k\varepsilon_k}\right)^{1/p_k}=\left(\frac{\omega_k(S_d)}{p\varepsilon }\right)^{1/p_k}>0,
\]
for each $k=1,\ldots,m$. For each $x\in\mathbb R^d$ which $|x|\geq1$, let
\[
S_x=\bigcap\limits_{k=1}^m\{t\in[0,1]^n:\;|s_k(t)x|>1\}.
\]
From the assumption 
$|s_k(t_1,\ldots,t_n)|\geq \min\{t_1^\beta,\ldots,t_n^\beta\}$ a.e $t=(t_1,\ldots,t_n)\in[0,1]^n$, there exists a null subset $E$ of $[0,1]^n$ so that $S_x$ contains $\left[1/|x|^{1/\beta},1\right]^n\setminus E$. From (\ref{sec2eq1}), we have
\begin{align*}
\;&\|U^{m,n}_{\psi,\overrightarrow{s}}\left(f_{p_1,\varepsilon_1},\ldots,f_{p_m,\varepsilon_m}\right)\|_{L^p_\omega\left(\mathbb R^d\right)}^p=\int_{\mathbb R^d}\left|\int_{[0,1]^n}\left(\prod\limits_{k=1}^mf_k\left(s_k(t)x\right)\right)\psi(t)dt\right|^p\omega(x)dx\\
=\;& \int_{\mathbb R^d}\prod_{k=1}^m|x|^{-\frac{d+\alpha_k}{p_k}-\varepsilon_k}\omega(x)\left|\int_{S_x}\prod_{k=1}^m|s_k(t)|^{-\frac{d+\alpha_k}{p_k}-\varepsilon_k}\psi(t)dt\right|^pdx\\
\geq\;& \int_{|x|\geq \varepsilon^{-\beta}}|x|^{-(d+\alpha+\varepsilon p)}\omega(x)\left|\,\int_{[\varepsilon,1]^n}\prod_{k=1}^m|s_k(t)|^{-\frac{d+\alpha_k}{p_k}-\varepsilon_k}\psi(t)dt\right|^pdx\\
=\;& \varepsilon^{p\beta\varepsilon}\left(\int_{|x|\geq1}|x|^{-(d+\alpha+\varepsilon p)}\omega(x)dx\right)\left|\,\int_{[\varepsilon,1]^n}\prod_{k=1}^m|s_k(t)|^{-\frac{d+\alpha_k}{p_k}-\varepsilon_k}\psi(t)dt\right|^p \\
=\;& \varepsilon^{p\beta\varepsilon}\left(\frac{\omega(S_d)}{\varepsilon p}\right)\left|\,\int_{[\varepsilon,1]^n}\prod_{k=1}^m|s_k(t)|^{-\frac{d+\alpha_k}{p_k}-\varepsilon_k}\psi(t)dt\right|^p \\
\geq\;& \varepsilon^{p\beta\varepsilon}\prod_{k=1}^m\|f_{p_k,\epsilon_k}\|_{p_k,\varepsilon_k}^{p/p_k}\left|\,\int_{[\varepsilon,1]^n}\prod_{k=1}^m|s_k(t)|^{-\frac{d+\alpha_k}{p_k}-\varepsilon_k}\psi(t)dt\right|^p. \\
\end{align*}
This implies that 
\begin{align*}
\left\|U^{m,n}_{\psi,\overrightarrow{s}}\right\|_{L^{p_1}_{\omega_1}\left(\mathbb R^d\right)\times\cdots\times L^{p_m}_{\omega_m}\left(\mathbb R^d\right)\to L^{p}_{\omega}\left(\mathbb R^d\right)}\geq  \varepsilon^{\beta\varepsilon}\cdot\left|\,\int_{[\varepsilon,1]^n}\prod_{k=1}^m|s_k(t)|^{-\frac{d+\alpha_k}{p_k}-\varepsilon_k}\psi(t)dt\right|.
\end{align*}
Notice that 
\[
|s_k(t)|^{-\varepsilon}\leq\min\{t_1,\ldots,t_n\}^{-\varepsilon\beta}\leq\varepsilon^{-\varepsilon\beta}\to1\quad\text{when $\varepsilon\to0^+$},
\]
 Thus, letting $\varepsilon\to0^+$ and by Lebesgue's dominated convergence theorem, we obtain
\begin{align}\label{sec3eq7}
\left\|U^{m,n}_{\psi,\overrightarrow{s}}\right\|_{L^{p_1}_{\omega_1}\left(\mathbb R^d\right)\times\cdots\times L^{p_m}_{\omega_m}\left(\mathbb R^d\right)\to L^{p}_{\omega}\left(\mathbb R^d\right)}\geq \mathcal A.
\end{align}
Combine (\ref{sec3eq6}) and (\ref{sec3eq7}), we obtain the result.

\end{proof}
Analogous to the proof of Theorem \ref{sec3theo1}, one can prove the following result. 
\begin{theorem}\label{sec3theo1a} Let $s_1,\ldots,s_m:[0,1]^n\to\mathbb R$ be measurable functions such that for every $k=1,\ldots,m$  then $|s_k(t_1,\ldots,t_n)|\geq \min\{t_1^\beta,\ldots,t_d^\beta\}$ almost everywhere $(t_1,\ldots,t_n)\in[0,1]^n$, $k=1,\ldots,m$, for some $\beta>0$. Let $1\leq p_k<\infty$, $k=1,\ldots,m$ and $\overrightarrow{s}=(s_1,\ldots,s_m)$. Assume that $(\omega_1,\ldots,\omega_m)$ satisfies $\mathcal W_{\overrightarrow{\alpha}}$ condition. Then $\mathcal H^{m,n}_{\psi,\overrightarrow{s}}$ is bounded from $L^{p_1}_{\omega_1}\left(\mathbb R^d\right)\times\cdots\times L^{p_m}_{\omega_m}\left(\mathbb R^d\right)$ to $L^{p}_{\omega}\left(\mathbb R^d\right)$ if and only if 
\begin{equation}\label{maineq1a}
{\mathcal A}_{\star}=\int_{\mathbb R_+^n}\left(\prod_{k=1}^m|s_k(t)|^{-\frac{d+\alpha_k}{p_k}}\right)\psi(t)dt<\infty.
\end{equation}
Furthermore, 
\begin{equation}\label{sec3eq4a}
\left\|\mathcal H^{m,n}_{\psi,\overrightarrow{s}}\right\|_{L^{p_1}_{\omega_1}\left(\mathbb R^d\right)\times\cdots\times L^{p_m}_{\omega_m}\left(\mathbb R^d\right)\to L^{p}_{\omega}\left(\mathbb R^d\right)}={\mathcal A}_\star.
\end{equation}
\end{theorem}

\begin{theorem}\label{sec3theo3} Let $1\leq p,p_k<\infty$, $ \lambda,\alpha_k,\lambda_k$ be real numbers such that $-\frac1{p_k}\leq \lambda_k<0$, $k=1,\ldots,m$  , $\frac1p=\frac1{p_1}+\cdots+\frac1{p_m}$,  and  $\lambda=\frac{d+\alpha_1}{d+\alpha}\lambda_1+\cdots+\frac{d+\alpha_m}{d+\alpha}\lambda_m$. 
Let $\omega_k\in {\mathcal W}_{\alpha_k}$ for $k=1,\ldots,m$.
\begin{enumerate}
\item[{(i)}] If in addition 
\begin{equation}\label{sec3eq8}
\left(\frac{\omega(S_d)}{d+\alpha}\right)^{\frac{1+\lambda p}p}\geq\prod_{k=1}^m\left(\frac{\omega_k(S_d)}{d+\alpha_k}\right)^{\frac{1+\lambda_k p_k}{p_k}}
\end{equation}
and 
\begin{equation}\label{sec3eq9}
{\mathcal B}=\int_{[0,1]^n}\left(\prod_{k=1}^m|s_k(t)|^{-\frac{(d+\alpha_k)\lambda_k}{p_k}}\right)\psi(t)dt<\infty,
\end{equation}
then $U^{m,n}_{\psi,\overrightarrow{s}}$ is bounded from $\dot{B}^{p_1,\lambda_1}_{\omega_1}\left(\mathbb R^d\right)\times\cdots\times \dot{B}^{p_m,\lambda_m}_{\omega_m}\left(\mathbb R^d\right)$ to   $\dot{B}^{p,\lambda}_{\omega}\left(\mathbb R^d\right)$. 
Furthermore, the operator norm of $U^{m,n}_{\psi,\overrightarrow{s}}$ not more than $\mathcal B$. 
\item[{(ii)}] Conversely, if 
\begin{equation}\label{sec3eq10}
\left(\frac{\omega(S_d)}{d+\alpha}\right)^{\lambda }\left(1+\lambda p\right)^{1/p}\leq\prod_{k=1}^m \left(\frac{\omega(S_d)}{d+\alpha_k}\right)^{\lambda_k }\left(1+\lambda_k p_k\right)^{1/p_k}
\end{equation}
and $U^{m,n}_{\psi,\overrightarrow{s}}$ is bounded from $\dot{B}^{p_1,\lambda_1}_{\omega_1}\left(\mathbb R^d\right)\times\cdots\times \dot{B}^{p_m,\lambda_m}_{\omega_m}\left(\mathbb R^d\right)$ to   $\dot{B}^{p,\lambda}_{\omega}\left(\mathbb R^d\right)$, then $\mathcal B$ is finite. Furthermore, we have
\begin{equation}\label{sec3eq11}
\left\|U^{m,n}_{\psi,\overrightarrow{s}}\right\|_{\dot{B}^{p_1,\lambda_1}_{\omega_1}\left(\mathbb R^d\right)\times\cdots\times \dot{B}^{p_m,\lambda_m}_{\omega_m}\left(\mathbb R^d\right)\to L^{p}_{\omega}\left(\mathbb R^d\right)}\geq {\mathcal B}.
\end{equation}
\end{enumerate}
\end{theorem}
Let $\omega_k\equiv 1$, we have that $\alpha_k=0$ thus (\ref{sec3eq8}) becomes equality and (\ref{sec3eq10}) holds only when $\lambda_1p_1=\cdots=\lambda_mp_m$. Hence, we obtain Theorem 2.1 in \cite{fu2}.

\begin{proof}
Since $\frac1p=\frac1{p_1}+\cdots+\frac1{p_m}$, by Minkowski's inequality and H\"{o}lder's inequality, we see that, for all balls $B=B(0,R)$,
\begin{align*}
\quad&\left(\frac1{\omega(B)^{1+\lambda p}}\int_B\left|U^{m,n}_{\psi,\overrightarrow{s}}\left(f_1,\ldots,f_m\right)(x)\right|^p\omega(x)dx\right)^{1/p}\\
\leq\;&\int_{[0,1]^n}\left(\frac1{\omega(B)^{1+\lambda p}}\int_B\left|\prod_{k=1}^mf_k\left(s_k(t)x\right)\right|^p\omega(x)dx\right)^{1/p}\psi(t)dt\\
\leq \;&\int_{[0,1]^n}\prod_{k=1}^m\left(\frac1{\omega_k(B)^{1+\lambda_k p_k}}\int_B\left|f_k\left(s_k(t)x\right)\right|^{p_k}\omega_k(x)dx\right)^{1/p_k}\psi(t)dt\\
\end{align*} 
\begin{align*}
\leq \;&\int_{[0,1]^n}\prod_{k=1}^m\left(\frac1{|s_k(t)B|^{1+\lambda_k p_k}}\int_B\left|f_k\left(y\right)\right|^{p_k}\omega_k(y)dy\right)^{1/p_k}\left(\prod_{k=1}^m|s_k(t)|^{(d+\alpha_k)\lambda_k }\right)\psi(t)dt\\
\leq \;&\prod_{k=1}^m\|f_k\|_{\dot{B}^{p_k,\lambda_k}_{\omega_k}\left(\mathbb R^d\right)}\cdot\mathcal B.\\
\end{align*} 
This means that 
\begin{equation}\label{sec3eq12}
\left\|U^{m,n}_{\psi,\overrightarrow{s}}\right\|_{\dot{B}^{p_1,\lambda_1}_{\omega_1}\left(\mathbb R^d\right)\times\cdots\times \dot{B}^{p_m,\lambda_m}_{\omega_m}\left(\mathbb R^d\right)\to \dot{B}^{p,\lambda}_{\omega}\left(\mathbb R^d\right)}\leq {\mathcal B}.
\end{equation}
Now we assume that (\ref{sec3eq10}) holds and $U^{m,n}_{\psi,s}$ is bounded from $\dot{B}^{p_1,\lambda_1}_{\omega_1}\left(\mathbb R^d\right)\times\cdots\times \dot{B}^{p_m,\lambda_m}_{\omega_m}$ to $ \dot{B}^{p,\lambda}_{\omega}\left(\mathbb R^d\right)$. Let $f_k(x)=|x|^{(d+\alpha_k)\lambda_k}$, then an elementary computation shows that $f_k\in\dot{B}^{p_k,\lambda_k}_{\omega_k}\left(\mathbb R^d\right)$ and 
\[
\|f_k\|_{\dot{B}^{p_k,\lambda_k}_{\omega_k}\left(\mathbb R^d\right)}=\left(\frac{d+\alpha_k}{\omega(S_d)}\right)^{\lambda_k}\left(1+\lambda_k p_k\right)^{-1/p_k}.
\]
Thus,
\begin{align*}
\prod_{k=1}^m\|f_k\|_{\dot{B}^{p_k,\lambda_k}_{\omega_k}\left(\mathbb R^d\right)}=\;&\prod_{k=1}^m\left(\frac{d+\alpha_k}{\omega(S_d)}\right)^{\lambda_k}\left(1+\lambda_k p_k\right)^{-1/p_k}\\
\leq\;&\left(\frac{d+\alpha }{\omega(S_d)}\right)^{\lambda}\left(1+\lambda p\right)^{-1/p}.
\end{align*}
This leads us to
\[
\|U^{m,n}_{\psi,\overrightarrow{s}}(f_1,\ldots,f_m)\|_{\dot{B}^{p,\lambda}_{\omega}\left(\mathbb R^d\right)}={\mathcal B}\cdot\left(\frac{d+\alpha }{\omega(S_d)}\right)^{\lambda}\left(1+\lambda p\right)^{-1/p}\geq {\mathcal B}\cdot \prod_{k=1}^m\|f_k\|_{\dot{B}^{p_k,\lambda_k}_{\omega_k}\left(\mathbb R^d\right)}.
\]
Therefore,
\begin{equation*}
\left\|U^{m,n}_{\psi,\overrightarrow{s}}\right\|_{\dot{B}^{p_1,\lambda_1}_{\omega_1}\left(\mathbb R^d\right)\times\cdots\times \dot{B}^{p_m,\lambda_m}_{\omega_m}\left(\mathbb R^d\right)\to L^{p}_{\omega}\left(\mathbb R^d\right)}\geq {\mathcal B}.
\end{equation*}
\end{proof}
From the proof of Theorem \ref{sec3theo3}, we also obtain the similar result to $\mathcal H^{m,n}_{\psi,s}$ operator.

\begin{theorem}\label{sec3theo3} Let $1\leq p,p_k<\infty$, $ \lambda,\alpha_k,\lambda_k$ be real numbers such that $-\frac1{p_k}\leq \lambda_k<0$, $k=1,\ldots,m$  , $\frac1p=\frac1{p_1}+\cdots+\frac1{p_m}$,  and  $\lambda=\frac{d+\alpha_1}{d+\alpha}\lambda_1+\cdots+\frac{d+\alpha_m}{d+\alpha}\lambda_m$. 
Let $\omega_k\in {\mathcal W}_{\alpha_k}$ for $k=1,\ldots,m$.
\begin{enumerate}
\item[{(i)}] If in addition 
\begin{equation}\label{sec3eq8}
\left(\frac{\omega(S_d)}{d+\alpha}\right)^{\frac{1+\lambda p}p}\geq\prod_{k=1}^m\left(\frac{\omega_k(S_d)}{d+\alpha_k}\right)^{\frac{1+\lambda_k p_k}{p_k}}
\end{equation}
and 
\begin{equation}\label{sec3eq9}
{\mathcal B}_\star=\int_{\mathbb R_+^n}\left(\prod_{k=1}^m|s_k(t)|^{-\frac{(d+\alpha_k)\lambda_k}{p_k}}\right)\psi(t)dt<\infty,
\end{equation}
then $\mathcal H^{m,n}_{\psi,\overrightarrow{s}}$ is bounded from $\dot{B}^{p_1,\lambda_1}_{\omega_1}\left(\mathbb R^d\right)\times\cdots\times \dot{B}^{p_m,\lambda_m}_{\omega_m}\left(\mathbb R^d\right)$ to   $\dot{B}^{p,\lambda}_{\omega}\left(\mathbb R^d\right)$. 
Furthermore, the operator norm of $\mathcal H^{m,n}_{\psi,\overrightarrow{s}}$ not more than $\mathcal B_\star$. 
\item[{(ii)}] Conversely, if 
\begin{equation}\label{sec3eq10}
\left(\frac{\omega(S_d)}{d+\alpha}\right)^{\lambda }\left(1+\lambda p\right)^{1/p}\leq\prod_{k=1}^m \left(\frac{\omega(S_d)}{d+\alpha_k}\right)^{\lambda_k }\left(1+\lambda_k p_k\right)^{1/p_k}
\end{equation}
and $\mathcal H^{m,n}_{\psi,\overrightarrow{s}}$ is bounded from $\dot{B}^{p_1,\lambda_1}_{\omega_1}\left(\mathbb R^d\right)\times\cdots\times \dot{B}^{p_m,\lambda_m}_{\omega_m}\left(\mathbb R^d\right)$ to   $\dot{B}^{p,\lambda}_{\omega}\left(\mathbb R^d\right)$, then $\mathcal B_\star$ is finite. Furthermore, we have
\begin{equation}\label{sec3eq11}
\left\|\mathcal H^{m,n}_{\psi,\overrightarrow{s}}\right\|_{\dot{B}^{p_1,\lambda_1}_{\omega_1}\left(\mathbb R^d\right)\times\cdots\times \dot{B}^{p_m,\lambda_m}_{\omega_m}\left(\mathbb R^d\right)\to L^{p}_{\omega}\left(\mathbb R^d\right)}\geq {\mathcal B}_\star.
\end{equation}
\end{enumerate}
\end{theorem}
\section{Commutators of weighted multilinear Hardy-Ces\`{a}ro operator}
We use some analogous tools to study a second set of problems related now to multilinear versions of the commutators of Coifman, Rochberg and Weiss \cite{coifman}
In this section, we consider the sharp estimates of the multilinear commutator generated by $U^{m,n}_{\psi,\overrightarrow{s}}$ with symbols in $CMO^q(\mathbb R^d)$. 
\begin{definition}\label{sec4def3}Let $m,n\in\mathbb N$, $\psi:[0,1]^n\to[0,\infty)$, $s_1,\ldots,s_m:[0,1]^n\to\mathbb R$, $b_1,\ldots,b_m$ be locally integrable functions on $\mathbb R^d$ and $f_1,\ldots,f_m:\mathbb R^d\to\mathbb C$ be measurable functions. The commutator of weighted multilinear Hardy-Ces\`{a}ro operator $U^{m,n}_{\psi,\overrightarrow{s}}$ is defined as
\begin{equation}\label{sec4eq1}
U^{m,n,\overrightarrow{b}}_{\psi,s}\left(f_1,\ldots,f_m\right)(x):=\int_{[0,1]^n}\left(\prod\limits_{k=1}^mf_k\left(s_k(t)x\right)\right)\left(\prod_{k=1}^m\left(b_k(x)-b_k\left(s_k(t)x\right)\right)\right)\psi(t)dt.
\end{equation}
\end{definition}
In what follows, we set 
\begin{equation}\label{sec4eq2}
{\mathcal C}=\int_{[0,1]^n}\left(\prod_{k=1}^m|s_k(t)|^{(d+\alpha_k)\lambda_k}\right) \psi(t)dt.
\end{equation}
\begin{equation}\label{sec4eq3}
{\mathcal D}=\int_{[0,1]^n}\left(\prod_{k=1}^m|s_k(t)|^{(d+\alpha_k)\lambda_k}\right)\left(\prod_{k=1}^m\left|\log|s_k(t)|\right|\right)\psi(t)dt.
\end{equation}
\begin{theorem}\label{sec4theo1} Let $1<p<p_k<\infty$, $1<q_k<\infty$, $-\frac1{p_k}<\lambda_k<0$, $k=1,\ldots,m$ such that
\[
\frac1p=\frac1{p_1}+\cdots+\frac1{p_m}+\frac1{q_1}+\cdots+\frac1{q_m}
\]
and $\lambda=\lambda_1+\cdots+\lambda_m$. 
\begin{enumerate}
\item[{(i)}] If both $\mathcal C$ and $\mathcal D$ are finite then for any $b=(b_1,\ldots,b_m)\in CMO^{q_1}_{\omega_1}\times CMO^{q_m}_{\omega_m}$ then $U^{m,n,\overrightarrow{b}}_{\psi,s}$ is bounded from $\dot{B}^{p_1,\lambda_1}_{\omega_1}(\mathbb R^d)\times\cdots\times \dot{B}^{p_m,\lambda_m}_{\omega_m}(\mathbb R^d)$ to $\dot{B}^{p,\lambda}_\omega(\mathbb R^d)$.
\item[{(ii)}]   If for any $b=(b_1,\ldots,b_m)\in CMO^{q_1}_{\omega_1}\times\cdots\times CMO^{q_m}_{\omega_m}$, $U^{m,n,\overrightarrow{b}}_{\psi,s}$ is bounded from $\dot{B}^{p_1,\lambda_1}_{\omega_1}(\mathbb R^d)\times\cdots\times \dot{B}^{p_m,\lambda_m}_{\omega_m}(\mathbb R^d)$ to $\dot{B}^{p,\lambda}_\omega(\mathbb R^d)$, then $\mathcal D$ is finite.
\end{enumerate}
\end{theorem}
We note here that $\mathcal D$ is finite is not enough to imply $\mathcal C$ is finite (see \cite{fu1,fu2}). But it is easy to see that, if we assume in addition that  for each $k=1,\ldots,m$ such that $|s_k(t)|\geq c>1$ for all $t\in[0,1]^n$ or $|s_k(t)|\leq c<1$ for all $t\in[0,1]^n$, then $\mathcal C$ is finite if and only if $\mathcal D$ is finite. Thus, Theorem \ref{sec4theo1}  implies immediately that
\begin{corollary}\label{sec4theo2} Let $1<p<p_k<\infty$, $1<q_k<\infty$, $-\frac1{p_k}<\lambda_k<0$, $k=1,\ldots,m$ such that
\[
\frac1p=\frac1{p_1}+\cdots+\frac1{p_m}+\frac1{q_1}+\cdots+\frac1{q_m}
\]
and $\lambda=\lambda_1+\cdots+\lambda_m$. Furthermore, suppose that for each $k=1,\ldots,m$ then $|s_k(t)|\geq c>1$ for all $t\in[0,1]^n$ or $|s_k(t)|\leq c<1$ for all $t\in[0,1]^n$, for each $k=1,\ldots,m$. Then,  $U^{m,n,\overrightarrow{b}}_{\psi,s}$ is bounded from $\dot{B}^{p_1,\lambda_1}_{\omega_1}(\mathbb R^d)\times\cdots\times \dot{B}^{p_m,\lambda_m}_{\omega_m}(\mathbb R^d)$ to $\dot{B}^{p,\lambda}_\omega(\mathbb R^d)$ if and only if $\mathcal C$ is finite.
\end{corollary}
Now we will give the proof for Theorem \ref{sec4theo1}.
\begin{proof}First we assume that $\mathcal C_m$ is finite. We shall give a details analysis for the case $m=2$ and by similarity, the general case is proved in the same way. We denote $B=B(0,R)$ for short. Let $\overrightarrow{b}=(b_1,b_2)\in CMO^{q_1}_{\omega_1}(\mathbb R^d)\times CMO^{q_2}_{\omega_2}(\mathbb R^d)$. By Minkowski's inequality, we have
\[
\left(\frac1{\omega\left(B\right)}\int_{B}\left|U^{2,n,b}_{\psi,s}\left(f_1,f_2\right)(x)\right|^p\omega(x)dx\right)^{1/p}
\]
\[
\leq\left(\frac1{\omega\left(B\right)}\int_{B}\left(\int_{[0,1]^n}\left(\prod\limits_{k=1}^2\left|f_k\left(s_k(t)x\right)\right|\right)\left(\prod_{k=1}^2\left|b_k(x)-b_k\left(s_k(t)x\right)\right|\right)\psi(t)dt\right)^p\omega(x)dx\right)^{1/p}
\]
\[
\leq\int_{[0,1]^n} \left(\frac1{\omega\left(B\right)}\int_{B}\left(\prod\limits_{k=1}^2\left|f_k\left(s_k(t)x\right)\right|\cdot\prod_{k=1}^2\left|b_k(x)-b_k\left(s_k(t)x\right)\right|\omega(x)\right)^pdx \right)^{1/p}\psi(t)dt
\]
\[
=:I.
\]
For any $ x_i,y_i,z_i,t_i\in\mathbb C$ with $i=1,2$, we have the following elementary inequality
\[
\prod_{i=1}^2\left|\left(x_i-y_i\right)\right|\leq\prod_{i=1}^2 \left|\left(x_i-z_i\right)\right|+\prod_{i=1}^2 \left|\left(y_i-t_i\right)\right|+\prod_{i=1}^2 \left|\left(z_i-t_i\right)\right|
\]
\[
+\left(\left|(x_1-z_1)(z_2-t_2)\right|+\left|(x_2-z_2)(z_1-t_1)\right|\right) +\left(\left|(x_1-z_1)(y_2-t_2)\right|+\left|(x_2-z_2)(y_1-t_1)\right|\right)
\]
\[
+\left(\left|(z_1-t_1)(y_2-t_2)\right|+\left|(z_2-t_2)(y_1-t_1)\right|\right). 
\]
 It is convenient to denote by $b_{i,\omega,B}$ the integrals $\int_{B}\frac1{\omega(B)}b_i(x)\omega(x)dx$ for $i=1,2$. Now applying the inequality with $x_i=b_i(x)$, $y_i=b_i(s_i(t)x)$, $z_i=b_{i,B}$, $t_i=b_{i,s_i(t)B}$ and using Minkowski's inequality, we get that 
\[
I\leq I_1+I_2+I_3+I_4+I_5+I_6,
\]
where, if we set $\overline{f}(x)=\prod\limits_{k=1}^2\left|f_k\left(s_k(t)x\right)\right|$, then
\[
I_1=\int_{[0,1]^n} \left(\frac1{\omega\left(B\right)}\int_{B}\left(\overline{f}(x)\cdot\prod_{k=1}^2\left|b_k(x)-b_{k,\omega_k,B}\right|\right)^p\omega(x)dx \right)^{1/p}\psi(t)dt,
\]
\[
I_2=\int_{[0,1]^n} \left(\frac1{\omega\left(B\right)}\int_{B}\left(\overline{f}(x)\cdot\prod_{k=1}^2\left|b_k(s_k(t)x)-b_{k,\omega_k,s_k(t)B}\right|\right)^p\omega(x)dx \right)^{1/p}\psi(t)dt,
\]
\[
I_3=\int_{[0,1]^n} \left(\frac1{\omega\left(B\right)}\int_{B}\left(\overline{f}(x)\cdot\prod_{k=1}^2\left|b_{k,\omega_k,B}-b_{k,\omega_k,s_k(t)B}\right|\right)^p\omega(x)dx \right)^{1/p}\psi(t)dt,
\]
\[
I_4=\int_{[0,1]^n} \left(\frac1{\omega\left(B\right)}\int_{B}\left(\overline{f}(x)\cdot \sum\limits_{i\neq j \atop i,j=1,2}\left|\left(b_i(x)-b_{i,B}\right)\left(b_{j,B}-b_{j,\omega_j,s_j(t)B}\right)\right| \right)^p\omega(x)dx \right)^{1/p}\psi(t)dt,
\]
\[
I_5=\int_{[0,1]^n} \left(\frac1{\omega\left(B\right)}\int_{B}\left(\overline{f}(x)\cdot \sum\limits_{i\neq j \atop i,j=1,2}\left|\left(b_i(x)-b_{i,B}\right)\left(b_j(s_j(t)x)-b_{j,\omega_j,s_j(t)B}\right)\right|\right)^p\omega(x)dx \right)^{1/p}\psi(t)dt,
\]
\[
I_6=\int_{[0,1]^n} \left(\frac1{\omega\left(B\right)}\int_{B}\left(\overline{f}(x)\cdot \sum\limits_{i\neq j \atop i,j=1,2}\left|\left(b_{i,B}-b_{i,s_i(t)B}\right)\left(b_j(s_j(t)x)-b_{j,\omega_j,s_j(t)B}\right)\right| \right)^p\omega(x)dx \right)^{1/p}\psi(t)dt.
\]
Choose now $p<s_1<\infty$ and $p<s_2<\infty$ such that $\frac1{s_1}=\frac1{p_1}+\frac1{q_1}$ and $\frac1{s_2}=\frac1{p_2}+\frac1{q_2}$. Notice that $\frac1{s_1}+\frac1{s_2}=\frac1p$. Then by H\"{o}lder's inequality, we have that
\begin{align*}
I_1 \leq& \int_{[0,1]^n} \prod\limits_{k=1}^2 \left(\frac1{\omega_k\left(B\right)}\int_{B}\left|f_k\left(s_k(t)x\right)\right|^{p_k}\omega_k(x)dx\right)^{1/p_k}\times\\ 
&\qquad\qquad\times \prod_{k=1}^2\left(\frac1{\omega_k\left(B\right)}\int_B\left|b_k(x)-b_{k,\omega_k,B}\right|^{q_k}\omega_k(x)\right)^{1/q_k}\psi(t)dt\\
 \leq& \int_{[0,1]^n}\prod_{k=1}^2|s_k(t)|^{(d+\alpha_k)\lambda_k}\omega_k(B)^{\lambda_k} \prod\limits_{k=1}^2 \left(\frac1{\omega_k\left(s_k(t)B\right)^{1+\lambda_k p_k}}\int_{s_k(t)B}\left|f_k\left(y\right)\right|^{p_k}\omega_k(y)dy\right)^{1/p_k}\times\\
 &\qquad\qquad\times\prod_{k=1}^2\left(\frac1{\omega_k\left(B\right)}\int_B\left|b_k(x)-b_{k,\omega_k,B}\right|^{q_k}\omega_k(x)\right)^{1/q_k}\psi(t)dt\\
 &\leq \prod_{k=1}^2\omega_k(B)^{\lambda_k}\prod_{k=1}^2\|b_k\|_{CMO^{q_k}_{\omega_k}}\prod_{k=1}^2\|f_k\|_{\dot{B}^{p_k,\lambda_k}}\times\int_{[0,1]^n}\left(\prod_{k=1}^2|s_k(t)|^{(d+\alpha_k)\lambda_k}\right)\psi(t)dt.
\end{align*}
Similarly to the estimate of $I_1$, we have that
\begin{align*}
I_2 \leq&  \int_{[0,1]^n} \prod\limits_{k=1}^2\left(\frac1{\omega_k\left(B\right)}\int_{B}\left|f_k\left(s_k(t)x\right)\right|^{p_k}\omega_k(x)dx \right)^{1/p_k}\times\\
&\qquad\qquad\times\prod_{k=1}^2\left(\frac1{\omega_k\left(B\right)}\int_{B}\left|b_k(s_k(t)x)-b_{k,\omega_k,s_k(t)B}\right|^{q_k}\omega(x)dx\right)^{1/q_k}\psi(t)dt\\
\end{align*}
\begin{align*}
\quad\leq&  \int_{[0,1]^n}\prod_{k=1}^2|s_k(t)|^{(d+\alpha_k)\lambda_k}\omega_k(B)^{\lambda_k}  \prod\limits_{k=1}^2\left(\frac1{\omega_k\left(B\right)^{1+\lambda_kp_k}}\int_{s_k(t)B}\left|f_k\left(y\right)\right|^{p_k}\omega_k(y)dy \right)^{1/p_k}\times\\
&\qquad\qquad\times\prod_{k=1}^2\left(\frac1{\omega_k\left(s_k(t)B\right)}\int_{s_k(t)B}\left|b_k(y)-b_{k,\omega_k,s_k(t)B}\right|^{q_k}\omega(y)dy\right)^{1/q_k}\psi(t)dt\\
&\leq \prod_{k=1}^2\omega_k(B)^{\lambda_k}\prod_{k=1}^2\|b_k\|_{CMO^{q_k}_{\omega_k}}\prod_{k=1}^2\|f_k\|_{\dot{B}^{p_k,\lambda_k}}\times\int_{[0,1]^n}\left(\prod_{k=1}^2|s_k(t)|^{(d+\alpha_k)\lambda_k}\right)\psi(t)dt.
\end{align*}
Now we give the estimate for $I_3$, applying H\"{o}lder's inequality, we get that
\begin{align*}
I_3=& \int_{[0,1]^n} \left(\frac1{\omega\left(B\right)}\int_{B}\left(\prod\limits_{k=1}^2\left|f_k\left(s_k(t)x\right)\right|\omega(x)\right)^pdx \right)^{1/p}\prod_{k=1}^2\left|b_{k,\omega_k,B}-b_{k,\omega_k,s_k(t)B}\right|\psi(t)dt  \\
\leq& \int_{[0,1]^n} \prod\limits_{k=1}^2\left(\frac1{\omega_k\left(B\right)}\int_{B}\left|f_k\left(s_k(t)x\right)\right|^{s_k}\omega_k(x)dx\right)^{1/s_k}\prod_{k=1}^2\left|b_{k,\omega_k,B}-b_{k,\omega_k,s_k(t)B}\right|\psi(t)dt  \\
\leq& \int_{[0,1]^n} \prod\limits_{k=1}^2\left(\frac1{\omega_k\left(s_k(t)B\right)^{1+p_k\lambda_k}}\int_{s_k(t)B}\left|f_k\left(y\right)\right|^{p_k}\omega_k(y)dy\right)^{1/p_k}\prod_{k=1}^2|s_k(t)|^{(d+\alpha_k)\lambda_k}\omega_k(B)^{\lambda_k}\times\\
&\qquad\qquad\times\prod_{k=1}^2\left|b_{k,\omega_k,B}-b_{k,\omega_k,s_k(t)B}\right|\psi(t)dt  \\
\leq& \prod\limits_{k=1}^2\omega_k(B)^{\lambda_k} \prod\limits_{k=1}^2\|f_k\|_{\dot{B}^{p_k,\lambda_k}}\int_{[0,1]^n} \prod_{k=1}^2|s_k(t)|^{(d+\alpha_k)\lambda_k} \prod_{k=1}^2\left|b_{k,\omega_k,B}-b_{k,\omega_k,s_k(t)B}\right|\psi(t)dt
\end{align*}
Notice that $[0,1]^n$ is the union of pairwise disjoint subsets $S_{\ell_1,\ldots,\ell_m}$, where
\[
S_{\ell_1,\ldots,\ell_m}=\bigcap\limits_{i=1}^m\left\{t\in[0,1]^n:\;2^{-\ell_i-1}<|s_i(t)|\leq 2^{-\ell_i}\right\}.
\]
Thus we obtain that
\begin{align*}
I_3\leq & \prod\limits_{k=1}^2\omega_k(B)^{\lambda_k} \prod\limits_{k=1}^2\|f_k\|_{\dot{B}^{p_k,\lambda_k}}\sum\limits_{\ell_1,\ell_2=0}^{\infty}\;\int_{S_{\ell_1,\ell_2}} \prod_{k=1}^2|s_k(t)|^{(d+\alpha_k)\lambda_k}\,\psi(t)\times\\
&\times \prod_{k=1}^2\left(\sum\limits_{j=0}^{\ell_k}\left|b_{k,\omega_k,2^{-j-1}B}-b_{k,\omega_k,2^{-j}B}\right|+\left|b_{k,\omega_k,2^{-\ell_k-1}B}-b_{k,\omega_k,s_k(t)B}\right|\right)dt\\
\leq&\; C\prod\limits_{k=1}^2\omega_k(B)^{\lambda_k} \prod\limits_{k=1}^2\|f_k\|_{\dot{B}^{p_k,\lambda_k}}\sum\limits_{\ell_1,\ell_2=0}^{\infty}\;\int_{S_{\ell_1,\ell_2}} \prod_{k=1}^2|s_k(t)|^{(d+\alpha_k)\lambda_k}\times\\
&\times \prod_{k=1}^2\left(\ell_k+2\right)\|b_k\|_{C\dot{M}O^{q_k}_{\omega_k}}\times\psi(t)dt\\
\leq&C\prod\limits_{k=1}^2\omega_k(B)^{\lambda_k} \prod\limits_{k=1}^2\|f_k\|_{\dot{B}^{p_k,\lambda_k}}\prod_{k=1}^2\|b_k\|_{C\dot{M}O^{q_k}_{\omega_k}} \int_{[0,1]^n} \prod_{k=1}^2\left(|s_k(t)|^{(d+\alpha_k)\lambda_k}\log\frac4{|s_k(t)|}\right)\psi(t)dt.
\end{align*}
Now we give the estimate for $I_4$. Similarly, we choose $1<s<\infty$ such that $\frac1s=\frac1{q_1}+\frac1{q_2}$. Let
\[
\overline{b_{i,j}}(x):=\left|\left(b_i(x)-b_{i,B}\right)\left(b_{j,B}-b_{j,\omega_j,s_j(t)B}\right)\right|,
\]
then Minkowski's inequality and H\"{o}lder’s inequality imply that

\begin{align*}
I_4= &  \int_{[0,1]^n} \left(\frac1{\omega\left(B\right)}\int_{B}\left(\left(\prod\limits_{k=1}^2\left|f_k\left(s_k(t)x\right)\right|\right)\left(\sum\limits_{i\neq j \atop i,j=1,2}\overline{b_{i,j}}(x)\right)\right)^p\omega(x)dx \right)^{1/p}\psi(t)dt\\
\leq & \int_{[0,1]^n} \sum\limits_{i\neq j \atop i,j=1,2} \left(\frac1{\omega\left(B\right)}\int_{B}\left(\left(\prod\limits_{k=1}^2\left|f_k\left(s_k(t)x\right)\right|\right)\cdot  \overline{b_{i,j}}(x)\right)^p\omega(x)dx \right)^{1/p}\psi(t)dt\\
\leq& \int_{[0,1]^n}\prod\limits_{k=1}^2\left(\frac1{\omega_k(B)}\int_{B}\left|f_k\left(s_k(t)x\right)\right|^{p_k}\omega_k(x)dx\right)^{1/p_k}\\
&\times\left( \sum\limits_{i\neq j \atop i,j=1,2}\left(\frac1{\omega_i\left(B\right)}\int_B\left|b_i(x)-b_{i,B}\right|^s\omega_i(x)dx\right)^{1/s}\left|b_{j,B}-b_{j,\omega_j,s_j(t)B}\right|\right)\\
 \leq&C \prod_{k=1}^2\omega_k\left(B\right)^{\lambda_k} \int_{[0,1]^n} \prod_{k=1}^2 |s_k(t)|^{(d+\alpha_k)\lambda_k}\prod\limits_{k=1}^2\left(\frac1{\omega_k(B)^{1+\lambda_k p_k}}\int_{s_k(t)B}\left|f_k\left(x\right)\right|^{p_k}\omega_k(x)dx\right)^{1/p_k}\\
&\times\left( \sum\limits_{i\neq j \atop i,j=1,2}\left(\frac1{\omega_i\left(B\right)}\int_B\left|b_i(x)-b_{i,B}\right|^s\omega_i(x)dx\right)^{1/s}\left|b_{j,B}-b_{j,\omega_j,s_j(t)B}\right|\right)\psi(t)dt\\
\leq& C\prod_{k=1}^2\omega_k\left(B\right)^{\lambda_k} \prod\limits_{k=1}^2\|f_k\|_{\dot{B}^{p_k,\lambda_k}}\int_{[0,1]^n} \prod_{k=1}^2 |s_k(t)|^{(d+\alpha_k)\lambda_k} \psi(t)\\
&\times\left( \sum\limits_{i\neq j \atop i,j=1,2}\left(\frac1{\omega_i\left(B\right)}\int_B\left|b_i(x)-b_{i,B}\right|^s\omega_i(x)dx\right)^{1/s}\left|b_{j,B}-b_{j,\omega_j,s_j(t)B}\right|\right)dt\\
\end{align*}
From the estimates of $I_1$ and $I_3$, we deduce that
\begin{align*}
I_4\leq&C\prod\limits_{k=1}^2\omega_k(B)^{\lambda_k} \prod\limits_{k=1}^2\|f_k\|_{\dot{B}^{p_k,\lambda_k}}\prod_{k=1}^2\|b_k\|_{C\dot{M}O^{q_k}_{\omega_k}} \\
&\times\int_{[0,1]^n} \prod_{k=1}^2|s_k(t)|^{(d+\alpha_k)\lambda_k}\left(1+\sum\limits_{k=1}^2\log\frac2{|s_k(t)|}\right)\psi(t)dt.
\end{align*}
It can be deduced from the estimates of $I_1,I_2,I_3,I_4$ that 
\begin{align*}
I_5\leq &C\prod\limits_{k=1}^2\omega_k(B)^{\lambda_k} \prod\limits_{k=1}^2\|f_k\|_{\dot{B}^{p_k,\lambda_k}}\prod_{k=1}^2\|b_k\|_{C\dot{M}O^{q_k}_{\omega_k}} \int_{[0,1]^n} \prod_{k=1}^2|s_k(t)|^{(d+\alpha_k)\lambda_k} \psi(t)dt.\\
I_6\leq& C\prod\limits_{k=1}^2\omega_k(B)^{\lambda_k} \prod\limits_{k=1}^2\|f_k\|_{\dot{B}^{p_k,\lambda_k}}\prod_{k=1}^2\|b_k\|_{C\dot{M}O^{q_k}_{\omega_k}}\\
&\times \int_{[0,1]^n} \prod_{k=1}^2|s_k(t)|^{(d+\alpha_k)\lambda_k} \left(1+\sum\limits_{k=1}^2\log\frac2{|s_k(t)|}\right)\psi(t)dt.\\
\end{align*}
Combining the estimates of $I_1,I_2,I_3,I_4,I_5$ and $I_6$ gives
\begin{align*}
&\left(\frac1{\omega\left(B\right)}\int_{B}\left|U^{2,n,b}_{\psi,s}\left(f_1,f_2\right)(x)\right|^p\omega(x)dx\right)^{1/p}\leq C\omega(B)^{\lambda} \prod\limits_{k=1}^2\|f_k\|_{\dot{B}^{p_k,\lambda_k}}\prod_{k=1}^2\|b_k\|_{C\dot{M}O^{q_k}_{\omega_k}}\\
&\qquad\times\int_{[0,1]^n} \prod_{k=1}^2\left(|s_k(t)|^{(d+\alpha_k)\lambda_k}\log\frac4{|s_k(t)|}\right)\psi(t)dt.\\
\end{align*}
This proves (i). 
\vskip12pt
Now we prove the necessity in (ii). Assume that $U^{2,n,b}_{\psi,s}$ is bounded from $\dot{B}^{p_1,\lambda_1}_{\omega_1}(\mathbb R^d)\times \dot{B}^{p_m,\lambda_m}_{\omega_m}(\mathbb R^d)$ to $\dot{B}^{p,\lambda}_\omega(\mathbb R^d)$. Firstly, we need the following lemmas.
\begin{lemma}\label{sec4lem1}
If $\omega\in\mathcal W_\alpha$ and has doubling property, then $\log|x|\in BMO(\omega)$.
\end{lemma}
This lemma was proved in \cite{chuong1}, for convenience of the reader, the proof to the lemma is presented here.
\begin{proof} To prove $\log|x|\in BMO(\omega)$, for any $x_0\in\mathbb R^d$ and $r>0$, we must find a constant $c_{x_0,r}$, such that $\frac1{\omega\left(B(x_0,r)\right)}\int_{|x-x_0|\leq r}\left|\log|x|-c_{x_0,r}\right|\omega(x)dx$ is uniformly bounded. Since  
\begin{align*}
&\frac1{\omega\left(B(x_0,r)\right)}\int\limits_{|x-x_0|\leq r}\left|\log|x|-c_{x_0,r}\right|\omega(x)dx\\
=&\frac{r^{\alpha+n}}{\omega\left(B(x_0,r)\right)}\int\limits_{|z-r^{-1}x_0|\leq1}\left|\log|z|+\log r-c_{x_0,r}\right|\omega(z)dz\\ 
=&\frac1{\omega\left(B(r^{-1}x_0,1)\right)}\int\limits_{|z-r^{-1}x_0|\leq1}\left|\log|z|+\log r-c_{x_0,r}\right|\omega(z)dz,
\end{align*}
we may take $c_{x_0,r}=c_{r^{-1}x_0,1}-\log r$, and so things reduce to the case that $r=1$ and $x_0$ is arbitrary. Let 
\[
A_{x_0}=\frac1{\omega\left(B(x_0,1)\right)}\int\limits_{|z-x_0|\leq1}\left|\log|z|-c_{x_0,1}\right|\omega(z)dz.
\]
If $|x_0|\leq2$, we take $c_{x_0,1}=0$, and observe that
\[
A_{x_0}\leq \frac1{\omega\left(B(x_0,1)\right)}\int\limits_{|z|\leq3}\log3\cdot\omega(z)dz= \log3\cdot\frac{\omega\left(B(0,3)\right)}{\omega\left(B(x_0,1)\right)}
\]
\[
\leq \log3\cdot \frac{\omega\left(B(x_0,6)\right)}{\omega\left(B(x_0,1)\right)}\leq C<\infty,
\]
where the last inequality comes from the assumption that $\omega$ has doubling property.
\vskip12pt
If $|x_0|\geq2$, take $c_{x_0,1}=\log |x_0|$. In this case, notice that
\[
A_{x_0}=  \frac1{\omega\left(B(x_0,1)\right)}\int\limits_{B(x_0,1)}\left|\log\frac{|z|}{|x_0|}\right|\omega(z)dz
\]
\[
\leq \frac1{\omega\left(B(x_0,1)\right)}\int\limits_{B(x_0,1)}\max\left\{\log\frac{|x_0|+1}{|x_0|}\,,\,\log\frac{|x_0|}{|x_0|-1}\right\}\cdot\omega(z)dz
\]
\[
\leq \max\left\{\log\frac{|x_0|+1}{|x_0|}\,,\,\log\frac{|x_0|}{|x_0|-1}\right\}\leq\log2.
\]
Thus $\log|x|$ belongs to $BMO_\omega\left(\mathbb R^d\right)$.
\end{proof} 
\begin{lemma}\label{sec4lem2}
Let $\omega\in\mathcal W_\alpha$, where $\alpha>-d$, $1<p<\infty$ and $-\frac1p<\lambda$. The the function $f_0(x)=|x|^{(d+\alpha)\lambda}$ belongs to $\dot{B}^{p,\lambda}_\omega$ and 
\[
\|f\|_{\dot{B}^{p,\lambda}_\omega}=\left(\omega_k(S_d)\right)^{-\lambda_k}\left(\frac1{(d+\alpha)(1+\lambda p)}\right)^{1/p}.
\]
\end{lemma}
Since the proof of this lemma is straightforward we omit it.   
Now we set $b_1(x)=b_2(x)=\log|x|$. Lemma \ref{sec4lem1} gives $b_1,b_2$ belong to $BMO_\omega\left(\mathbb R^d\right)$. Since $BMO_\omega\left(\mathbb R^d\right) \subset C\dot{M}O^q_\omega\left(\mathbb R^d\right)$ for any $1<q<\infty$, we obtain that $b_k\in C\dot{M}O^{q_k}_{\omega_k}\left(\mathbb R^d\right)$. Define $f_k(x)=|x|^{(d+\alpha_k)\lambda_k}$ if $x\in\mathbb R^d\setminus\{0\}$ and $f_k(0)=0$.	From lemma \ref{sec4lem2}, we get that 
\begin{equation}
\|f_k\|_{\dot{B}^{p_k,\lambda_k}_{\omega_k}}=\left(\omega_k(S_d)\right)^{-\lambda_k}\left(\frac1{(d+\alpha_k)(1+\lambda_k p_k)}\right)^{1/p_k}.
\end{equation}
Also we have
\[
U^{2,n,b}_{\psi,s}\left(f_1,f_2\right)(x)=\prod\limits_{k=1}^2|x|^{(d+\alpha_k)\lambda_k}\int_{[0,1]^n}\prod\limits_{k=1}^2\left(|s_k(t)|^{(d+\alpha_k)\lambda_k}\log\frac1{|s_k(t)|}\right)dt.
\]
Let $B=B(0,R)$ be any ball of $\mathbb R^d$, then
\begin{align*}
&\left(\frac1{\omega(B)^{1+\lambda p}}\int_B\left|U^{2,n,b}_{\psi,s}\left(f_1,f_2\right)(x)\right|^p\omega(x)dx\right)^{1/p}\\
\leq \quad&\left(\frac1{\omega(B)^{1+\lambda p}}\int_B\left|x\right|^{(d+\alpha)\lambda p}\omega(x)dx\right)^{1/p}\left(\int_{[0,1]^n}\prod\limits_{k=1}^2\left(|s_k(t)|^{(d+\alpha_k)\lambda_k}\log\frac1{|s_k(t)|}\right)\psi(t)dt\right) \\
=\quad&\left(\omega_k(S_d)\right)^{-\lambda_k}\left(\frac1{(d+\alpha)(1+\lambda p)}\right)^{1/p}\left(\int_{[0,1]^n}\prod\limits_{k=1}^2\left(|s_k(t)|^{(d+\alpha_k)\lambda_k}\log\frac1{|s_k(t)|}\right)\psi(t)dt\right) \\
=\quad&\prod_{k=1}^2\|f_k\|_{\dot{B}^{p_k,\lambda_k}_{\omega_k}}\left(\int_{[0,1]^n}\prod\limits_{k=1}^2\left(|s_k(t)|^{(d+\alpha_k)\lambda_k}\log\frac1{|s_k(t)|}\right)\psi(t)dt\right).
\end{align*}
Taking the supremum over $R>0$, we get that 
\[
\int_{[0,1]^n}\prod\limits_{k=1}^2\left(|s_k(t)|^{(d+\alpha_k)\lambda_k}\log\frac1{|s_k(t)|}\right)\psi(t)dt<\infty.
\]
 
\end{proof}
\vskip12pt
{\bf Acknowledgements.} The authors would like to thank Prof. Nguyen Minh Chuong for many helpful suggestions and discussions. This paper is granted by Vietnam NAFOSTED (National Foundation for Science and Technology Development).

\end{document}